\documentclass[12pt]{amsart}
\usepackage{amssymb}

\numberwithin{equation}{section}

\begin{document}

{\theoremstyle{theorem}
    \newtheorem{theorem}{\bf Theorem}[section]
    \newtheorem{proposition}[theorem]{\bf Proposition}
    \newtheorem{conjecture}[theorem]{\bf Conjecture}
    \newtheorem{claim}{\bf Claim}[theorem]
    \newtheorem{lemma}[theorem]{\bf Lemma}
    \newtheorem{corollary}[theorem]{\bf Corollary}
}
{\theoremstyle{remark}
    \newtheorem{remark}[theorem]{\bf Remark}
    \newtheorem{example}[theorem]{\bf Example}
}
{\theoremstyle{definition}
    \newtheorem{definition}[theorem]{\bf Definition}
    \newtheorem{question}[theorem]{\bf Question}
}

\def\C{{\mathcal C}}
\def\H{{\mathcal H}}
\def\x{{\bf x}}
\def\y{{\bf y}}
\def\z{{\bf z}}
\def\w{{\bf w}}
\def\u{{\bf u}}
\def\height{\operatorname{ht}}

\title{Cohen-Macaulay Admissible Clutters}
 
\author{Huy T\`{a}i H\`{a}}
\address{Tulane University\\
Department of Mathematics\\
6823 St. Charles Avenue\\
New Orleans, LA 70118 }
\email{tai@math.tulane.edu}
\urladdr{http://www.math.tulane.edu/$\sim$tai/}

\author{Susan Morey}
\address{Department of Mathematics \\
Texas State University\\
601 University Drive\\ 
San Marcos, TX 78666}
\email{morey@txstate.edu}
\urladdr{http://www.txstate.edu/$\sim$sm26/}

\author{Rafael H. Villarreal}
\address{Departamento de Matem\'aticas\\
Centro de Investigaci\'on y de Estudios
Avanzados del IPN\\
Apartado Postal 14--740 \\
07000 Mexico City, D.F.}
\email{vila@math.cinvestav.mx}

\thanks{The first author is partially supported by Louisiana Board of Regents Grant LEQSF(2007-10)-RD-A-30 and Tulane Research Enhancement Fund. The third author acknowledges the financial support of CONACyT grant 49251-F and SNI}

\begin{abstract}
There is a one-to-one correspondence between square-free monomial ideals and clutters, which are also known as simple hypergraphs. In \cite{MRV} it was conjectured that unmixed admissible clutters were Cohen-Macaulay. We prove that the conjecture is true for uniform clutters of heights $2$ and $3$, i.e., if the smallest cardinality of a minimal vertex cover of the clutter is $2$ or $3$. For clutters of greater height, we give a family of counterexamples to show that the conjecture fails. For unmixed admissible uniform clutters of height $4$, we characterize when the Alexander dual of their edge ideals has linear quotients, and in particular, give an additional condition under which unmixed admissible uniform clutters are Cohen-Macaulay. 
\end{abstract}

\maketitle

\section{Introduction}

A {\bf clutter} consists of a finite set of points, called the {\bf vertices}, and a family of nonempty subsets of the vertices with no nontrivial containments, called the {\bf edges}. Clutters are also known as {\bf simple hypergraphs}. A basic example of a clutter is a simple graph in the classical sense. Throughout the paper, $\C$ will denote a clutter over the vertices $V(\C) = \{x_1, \dots, x_n\}$ and edges $E(\C)$.

Let $K$ be a field. By identifying the vertices $\{x_1, \dots, x_n\}$ with the variables of a polynomial ring $R = K[x_1, \dots, x_n]$, there is a natural one-to-one correspondence between the class of clutters over the vertices $\{x_1, \dots, x_n\}$ and the class of square-free monomial ideals in $R$. This correspondence is given by $\C \leftrightarrow I(\C)$, where $I(\C)$ is the ideal $\big\langle x^e = \prod_{x_i \in e} x_i ~\big|~ e \in E(\C) \big\rangle$ in $R$. The ideal $I(\C)$ is usually referred to as the {\bf edge ideal} of $\C$. Edge ideals of clutters can also be viewed as edge ideals of hypergraphs (cf. \cite{HVT}) or facet ideals of simplicial complexes (cf. \cite{Faridi}).

We say that $\C$ is a {\bf Cohen-Macaulay clutter} if $R/I(\C)$ is a Cohen-Macaulay ring. The goal of this paper is to determine classes of Cohen-Macaulay clutters. That is, we seek to describe classes of Cohen-Macaulay square-free monomial ideals.

A Cohen-Macaulay ring is always unmixed. Thus, we shall focus on clutters with this property. A subset $C$ of $V(\C)$ is called a {\bf vertex cover} of $\C$ if for every edge $e \in E(\C)$, we have $C \cap e \not= \emptyset$. A vertex cover for which no proper subset is also a vertex cover is called a {\bf minimal vertex cover}. Observe that
$$I(\C) = \bigcap_{\{x_{i_1}, \dots, x_{i_s}\} \text{ is a minimal vertex cover of } \C} (x_{i_1}, \dots, x_{i_s}).$$
From this, it follows that the smallest cardinality of a minimal vertex cover of $\C$, called the {\bf covering number}, is exactly the height, $\height I(\C)$, of $I(\C)$. By abuse of terminology, we shall also call $\height I(\C)$ the {\bf height} of $\C$. We say that $\C$ is {\bf unmixed} if all its minimal vertex covers have the same cardinality. Recall also that $\C$ is {\bf uniform} if all its edges have the same cardinality. A {\bf perfect matching} of $\C$ is a collection of pairwise disjoint edges of $\C$ whose union is exactly $V(\C)$. A perfect matching is said to be {\bf of K\"onig type} if it has $\height I(\C)$ edges.

Inspired by \cite{herzog-hibi}, where the Cohen-Macaulay property was studied for bipartite graphs with a perfect matching, the following notion of {\bf admissible} clutters was introduced in \cite{MRV}. 

\begin{definition} \label{def.admissible}
Let $\C$ be a clutter with $\height I(\C) = g$, and let $d$ be a positive integer. Suppose that there are two partitions $\{X^1, \dots, X^d\}$ and $\{e_1, \dots, e_g\}$ of $V(\C)$ such that $|e_i\cap X^j| \le 1$ for all $i,j$. 
\begin{enumerate}
\item Suppose $e$ is a subset of $V(\C)$ of size $k$ such that $|e\cap X^i|\leq 1$ for all $i$. Let $1 \leq i_1 < \dots < i_k \leq d$ be all the integers such that $e \cap X^{i_l} \not= \emptyset$. Let $j_1, \dots, j_k \in \{1, \dots, g\}$ be integers such that $e \cap X^{i_l} \in e_{j_l}$. We denote by $x^{i_l}_{j_l}$ the unique vertex of $e \cap X^{i_l} \cap e_{j_l}$. We say that $e$ is {\bf admissible} if $i_1=1, i_2=2, \dots, i_k=k$ and  $j_1 \leq \dots \leq j_k$. Such an admissible set can be represented as $e=x^1_{j_1} \cdots x^{k}_{j_k}$. 
\item A monomial $x^a$ is {\bf admissible} if ${\rm supp}(x^a)$ is admissible. 
\item We say that $\C$ is {\bf admissible} if $e_1, \dots, e_g \in E(\C)$, and all edges of $\C$ are admissible.
\end{enumerate}
We can think of $X^1, \dots, X^d$ as {\bf color classes} used to color the edges of $\C$. Note that under condition (3), $e_1, \ldots , e_g$ form a perfect matching of K\"{o}nig type in $\C$.
\end{definition}

In order to generalize results of \cite{herzog-hibi} to higher dimension, the following conjecture was stated in \cite{MRV}.   

\begin{conjecture}\label{conjecture} If $\C$ is an admissible clutter and $\C$ is unmixed, then $\C$ is Cohen-Macaulay.  
\end{conjecture}

Conjecture \ref{conjecture} is true for clutters with two color classes, i.e., when $d=2$ (see \cite{herzog-hibi}). It is also true if the admissible clutter $\C$ is uniform and {\it complete}, meaning that every maximal admissible set in $V(\C)$ is an edge of $\C$ (see \cite[Theorem 3.12]{MRV}). Note that if $\C$ is complete and admissible then $\C$ is automatically unmixed (see \cite[Theorem 3.6]{MRV}).

In this paper, we work on clutters with an arbitrary number of color classes. Our main results are as  follows. First, we give an affirmative answer to Conjecture \ref{conjecture} for uniform clutters when $\height I(\C) =2$ and when $\height I(\C) =3$ (Theorems \ref{g=2quotients} and \ref{g=3quotients}). Then we present a family of examples to show that Conjecture \ref{conjecture} may fail when $\height I(\C) \ge 4$, even in the uniform case (Theorem \ref{counterexamples}). When $\height I(\C) = 2$, we also show that if $I(\C)$ is {\bf normally torsion-free}, meaning all symbolic powers of $I(\C)$ are the same as the usual powers, then the converse statement of Conjecture \ref{conjecture} is true (Theorem \ref{g=2converse}). That is, if $\C$ is uniform and $I(\C)$ is Cohen-Macaulay, normally torsion-free, and of height two, then $\C$ is unmixed and has a perfect matching of K\"onig type and a partition for which it is admissible. Furthermore, when $\height I(\C) = 4$, we give an additional condition under which admissible unmixed clutters will be Cohen-Macaulay (Theorem \ref{thm.g=4}).

Our tool in examining Cohen-Macaulay clutters is the theory of Alexander duality. The {\bf Alexander dual} of the square-free monomial ideal $I(\C)$ is defined to be
\begin{align*}
I(\C)^\vee & = \bigcap_{x_{j_1} \cdots x_{j_r} \text{ is a minimal generator of } I(\C)} (x_{j_1}, \dots, x_{j_r}) \\
& = \big\langle x_{i_1} \cdots x_{i_s} ~\big|~ \{x_{i_1}, \dots, x_{i_s}\} \text{ is a minimal vertex cover of } \C \big\rangle.
\end{align*}
The Alexander dual of $I(\C)$ is also a square-free monomial ideal. We shall define the {\bf Alexander dual} of $\C$ to be the clutter corresponding to $I(\C)^\vee$, denoted by $\C^\vee$. Our method is based on the following theorem of Eagon and Reiner (see \cite{ER}).

\begin{theorem} \label{thm.er}
Let $I$ be a square-free monomial ideal in $R = K[x_1, \dots, x_n]$ with the Alexander dual $I^\vee$. Then $R/I$ is a Cohen-Macaulay ring over $K$ if and only if $I^\vee$ has a linear resolution over $R$.
\end{theorem}

Theorem \ref{thm.er} allows us to study the Cohen-Macaulayness of $I(\C)$ by investigating when its Alexander dual $I(\C)^\vee$ has a linear resolution. Proving that a class of ideals has linear resolutions is difficult in general. To do this, we shall employ techniques from Herzog and Takayama's theory of {\bf linear quotients} (see \cite{HT}).

\begin{definition} \label{def.linearquotients}
Let $I$ be a monomial ideal in $R = K[x_1, \dots, x_n]$. The ideal $I$ is said to have {\it linear quotients} if $I$ has a system of minimal generators $\{u_1, \dots, u_r\}$ with $\deg u_1 \le \dots \le \deg u_r$ such that for all $1 \le i \le r-1$, $((u_1, \dots u_i) : u_{i+1})$ is generated by linear forms.
\end{definition}

It can be seen that if a monomial ideal $I$ is generated in a single degree and $I$ has linear quotients, then $I$ has a linear resolution (cf. \cite[Lemma 5.2]{Faridi}). Thus, if $\C$ is unmixed then to show that $\C$ is Cohen-Macaulay, it suffices to show that $I(\C)^\vee$ has linear quotients.

The paper is outlined as follows. In the next section, we prove some auxiliary results about the generators of the Alexander dual $I(\C)^\vee$ of $I(\C)$. Section ~\ref{g=2} is devoted to the case when $\height I(\C) = 2$. Section ~\ref{g=3} deals with the case when $\height I(\C) = 3$. In Section ~\ref{higherg}, we give a family of counterexamples to Conjecture \ref{conjecture} when $\height I(\C) \ge 4$ and give a criterion for the Cohen-Macaulayness of $I(\C)$ when it has height four.


\section{Generators of the Alexander Dual} \label{lemmas} 

Throughout this section $\C$ will denote a uniform admissible clutter. We shall prove a number of auxiliary results about the Alexander dual $I(\C)^\vee$ of the edge ideal $I(\C)$ of $\C$. Recall that the generators of $I(\C)^\vee$ are identified with the minimal vertex covers of $\C$. Results in this section allow us to recognize which subsets of the vertices are vertex covers of $\C$. More precisely, these results allow us to produce additional minimal vertex covers from known ones. 

From Definition \ref{def.admissible}, $\C$ admits a perfect matching $\{e_1, \dots, e_g\}$, where $g$ denotes the height of $I(\C)$, and has a partition of the vertices $\{X^1, \dots, X^d\}$ (color classes). Since $\C$ is uniform and admissible, it can be seen that $|e_j| = d$ for all $j$, and $|e_j \cap X^i| = 1$ for all $i,j$ (cf. \cite{MRV}). As before, we use $x^i_j$ to denote the unique vertex in $e_j \cap X^i$. Also, we sometimes refer to the index $i$ in the vertex $x^i_j$ as its {\bf exponent}. Throughout the paper, we will be dealing with square-free monomials, so this notion of exponent will not cause any confusion with the {\it exponent} that refers to powers in a monomial (since the latter is always 1). 

\begin{lemma}\label{RaiseTheEnd}
Suppose $C$ is a minimal vertex cover of size $g$ of $\C$. If $C \cap e_i = \{x^t_i\}$ and $C \cap e_j = \{x^l_j\}$ for all $g \geq j >i$ for some fixed $l$,  $d \geq l>t$, then $C'=C\setminus \{x_i^t\} \cup \{x_i^{t+1}\}$ is also a minimal vertex cover of $\C$.
\end{lemma}

\begin{proof}
Note that $|C'| = |C| = g = \height I(\C)$, so if $C'$ is a cover, then it is necessarily minimal. Let $e$ be an arbitrary edge of $\C$. Then $e \cap C\not= \emptyset$. If $e \cap C$ contains any element other than $x_i^t$, then $e \cap C' \not= \emptyset$ and we are done. Thus we may assume $e \cap C = \{x^t_i\}$. Now consider $e \cap X^{t+1}= \{x^{t+1}_j\}$ for some $j\geq i$. If $j=i$, then $x^{t+1}_i \in C'$ and $C'$ covers $e$. If $j>i$, then since $l \geq t+1$ we have $e \cap X^l = \{x^l_k\}$ for some $k\geq j$. But then $x_k^l \in C'$ and $C'$ covers $e$. 
Thus $C'$ covers $e$ for any edge $e$ of $\mathcal C$ as desired.
\end{proof}

\begin{remark} In Lemma \ref{RaiseTheEnd}, if $i=g$ then the condition, in fact, is: $C \cap e_g = \{x^t_g\}$ where $t < d$. Thus, for any given minimal vertex cover of size $g$, we can create a family of minimal vertex covers of size $g$ by {\it raising the last term}. 
\end{remark}

Note that there is a symmetry to the definition of an admissible clutter. Using this symmetry, we can prove that the exponents for early terms can be reduced in a way that is symmetric to the method given in the preceding argument.

\begin{lemma}\label{LowerTheFront}
Suppose $C$ is a minimal vertex cover of size $g$ of $\C$. If $C\cap e_i = \{x^t_i\}$ and $C\cap e_j = \{x^l_j\}$ for all $1 \leq j < i$ and some fixed $l$,  $1 \leq l<t$, (or $1=i$ and $1<t$) then $C' = C \setminus \{x_i^t\} \cup \{x_i^{t-1}\}$ is also a minimal vertex cover of $\C$.
\end{lemma}

\begin{proof}
Let $e$ be an arbitrary edge of $\C$. As before, we need only show that $C'$ covers $e$. We may assume $e \cap C = \{x^t_i \}$. Consider $e \cap X^{t-1} = \{x_r^{t-1}\}$ for some $r\leq i$. If $r=i$, then $C'$ covers $e$. If $r <i$, consider $e \cap X^l = \{ x_s^l \}$ for some $s\leq r<i$. But then $x_s^l \in C'$ and so $C'$ covers $e$.
\end{proof}

\begin{remark} Lemma \ref{LowerTheFront} allows us to obtain a family of minimal vertex covers from a given one by {\it lowering the front term}.
\end{remark}

The final lemma of this section gives a method that can sometimes be used to alter a middle term of a vertex cover. It shows that if two vertex covers are identical except for their intersection with a fixed $e_i$ from the perfect matching of K\"{o}nig type, then one can form a family of minimal vertex covers, differing only in their intersections with the fixed $e_i$, that in some sense connects the two covers along $e_i$. 

\begin{lemma} \label{Consecutive}
Let $\C$ be an admissible clutter. Let $i$ and $c < c'$ be positive integers, and suppose that $C$ is a subset of the vertices of $\C$ such that both $C \cup \{x_i^c\}$ and $C \cup \{x_i^{c'}\}$ are vertex covers of $\C$. Then $C \cup \{x_i^l\}$ is a vertex cover of $\C$ for all $c \le l \le c'$.
\end{lemma}

\begin{proof} It follows from the hypothesis that any edge $e$ of $\C$ that avoids the vertices in $C$ must contain both $x_i^c$ and $x_i^{c'}$. That is, $e \cap X^c = \{x_i^c\}$ and $e \cap X^{c'} =\{x_i^{c'} \}$. 
Since $\C$ is admissible, this implies that $e \cap X^l = \{x_i^l \}$  for any $c \le l \le c'$. Hence, any edge $e$ of $\C$ that avoids the vertices in $C$ must contain $x_i^l$ for all $c \le l \le c'$. This proves the lemma.
\end{proof}





To conclude this section, we observe that if $C$ is a minimal vertex cover of $\C$ of size $g$, then we must have $|C \cap e_j| = 1$ for all $j$. Thus, the minimal generator $x^C$ of the Alexander dual $I(\C)^\vee$ can be written as $x^C = x_1^{i_1}x_2^{i_2} \dots x_g^{i_g}$ for $1 \leq i_j \leq d$. We, therefore, can work with the {\bf exponent vector} $(i_1, i_2, \dots ,i_g)$ when discussing $x^C$.


\section{Cohen-Macaulay Clutters of Height Two} \label{g=2}

This section is devoted to investigating the situation of uniform clutters with $g = \height I(\C) = 2$. We prove Conjecture \ref{conjecture} in this case. We also show that the converse statement of Conjecture \ref{conjecture} is true when the ideal is normally torsion-free. That is, if $\C$ is uniform and its edge ideal is normally torsion-free and Cohen-Macaulay of height two, then $\C$ is unmixed and has a perfect matching of K\"onig type and a partition (color classes) for which all edges of $\C$ are admissible.

As mentioned in the Introduction, to show that $\C$ is Cohen-Macaulay, it suffices to give an ordering of the minimal generators of $I(\C)^\vee$ so that it admits linear quotients. In the case when $g = 2$, our method is as follows. We first describe an ordering on the set of all tuples $S = \{ (a,b) ~|~ 1 \le a,b \le d\}$. This induces an ordering on the exponent vectors of the minimal generators of $I(\C)^\vee$. The minimal generators of $I(\C)^\vee$ are then labeled by the increasing order of their exponents. To show that under this ordering $I(\C)^\vee$ admits linear quotients, we show that if $V$ and $W$ are minimal generators of $I(\C)^\vee$ with $V$ labeled before $W$ and $(V:W)$ is not generated by a linear form (in this case, it means $V/gcd(V,W)$ is a monomial of degree greater than 1) then there exists a generator $U$ of $I(\C)^\vee$ labeled before $W$ such that $U/\gcd(U,W)$ is a linear factor of $V/\gcd(V,W)$.

\begin{theorem}\label{g=2quotients}
Let $\C$ be a uniform admissible unmixed clutter such that $g = \height I(\C) = 2$. Then the Alexander dual $I(\C)^\vee$ of $I(\C)$ has linear quotients, and so $\C$ is a Cohen-Macaulay clutter.
\end{theorem}

\begin{proof} Since $\C$ is unmixed, it will be Cohen-Macaulay if $I(\C)^\vee$ has linear quotients, as noted in the Introduction. To prove $I(\C)^\vee$ has linear quotients, we order the elements of $S$ by $(1,d) < (1,d-1) < \dots < (1,1) < (2,d) < (2,d-1) \dots < (2,1) < (3,d) < \dots < (d,d) < \dots < (d,1)$, and as mentioned before, label the minimal generators of $I(\C)^\vee$ according to the increasing order of their exponent vectors induced by the ordering on $S$. Assume that the minimal generators of $I(\C)^\vee$ are labeled as $u_1, \dots, u_s$. 

Now suppose that for some $j < i$, $\deg u_j/\gcd(u_j,u_i) \ge 2$. Write $u_j = x^{j_1}_1x^{j_2}_2$ and $u_i = x^{i_1}_1x^{i_2}_2$. Then by the chosen ordering, we have $j_1 < i_1$. By applying Lemma \ref{LowerTheFront} to $u_i$, we get a minimal generator $u_k = x_1^{j_1}x_2^{i_2}$ of $I(\C)^\vee$. It follows from the chosen ordering that $k < i$. Moreover, $u_k/\gcd(u_k,u_i) = x_1^{j_1}$, which divides $u_j/\gcd(u_j,u_i)$. Hence, under this labeling of the generators, $I(\C)^\vee$ has linear quotients. 
\end{proof}

\begin{remark} In the proof of Theorem \ref{g=2quotients}, we can apply Lemma \ref{RaiseTheEnd} to $u_j$ instead of using Lemma \ref{LowerTheFront}. There are other orderings of $S$ that also give rise to linear quotients in $I(\C)^\vee$. For example, it is easy to check that reverse lexicographical ordering in $S$ yields linear quotients in $I(\C)^\vee$. In passing to $g\geq 3$, it, however, becomes important to be able to use both Lemmas \ref{RaiseTheEnd} and \ref{LowerTheFront}. The ordering chosen in the proof above is designed to allow us to use both Lemma \ref{RaiseTheEnd} and Lemma \ref{LowerTheFront} to obtain generators of $I(\C)^\vee$ labeled before $u_i$.
\end{remark}

\begin{remark} It follows from \cite[Theorem 3.2]{HHZ} that, under the hypotheses of Theorem \ref{g=2quotients}, all powers of $I(\C)^\vee$ have linear resolutions.
\end{remark}

We now prove the converse of Theorem \ref{g=2quotients} under the additional assumption that $I(\C)$ is normally torsion free. Observe that if $\C$ is an unmixed clutter of height 2 then the Alexander dual $\C^\vee$ is a graph in the classical sense. Before proving the theorem, we shall recall the notions of a {\bf chordal} graph and of a {\bf free vertex} in a clutter.

\begin{definition} \label{def.chordal}
A graph $G$ is called a {\bf chordal} graph if every cycle of length at least 4 in $G$ has a {\bf chord}, that is, an edge joining two nonadjacent vertices of the cycle.
\end{definition}

\begin{remark} \label{dirac}
An alternative characterization of chordal graphs, due to Dirac \cite{dirac}, can be found in \cite{PTW}. It states that a graph $G$ is chordal if and only if every induced subgraph $H$ of $G$ contains a vertex $z$ such that the induced subgraph of $H$ on $N_H(z)$, the set of vertices adjacent to $z$ in $H$, is a complete subgraph of $H$. A vertex in $G$ with this property is called a {\bf simplicial vertex}.
\end{remark}

\begin{definition} \label{def.freevertex}
Let $\C$ be a clutter. Then a vertex $x \in V(\C)$ is called a {\bf free vertex} of $\C$ if $x$ belongs to exactly one edge of $\C$.
\end{definition}

\begin{theorem} \label{g=2converse}
Let $\C$ be a $d$-uniform clutter. Assume that $I(\C)$ is
normally torsion free and of height two. Then $\C$ is Cohen-Macaulay if and only if 
\begin{itemize}
\item[\rm (i)] $\C$ is unmixed, and
\item[\rm (ii)] there is a partition
$X^1=\{x_1^1,x_2^1\},\ldots,X^d=\{x_1^d,x_2^d\}$ 
of $V(\C)$ and a perfect matching $e_1=\{x_1^1,\ldots,x_1^d\}$, 
$e_2=\{x_2^1,\ldots,x_2^d\}$ of $\C$ such that all edges of
$\C$ have the form $\{x_{i_1}^1,\ldots,x_{i_d}^d\}$ for some
$1\leq i_1\leq\cdots\leq i_d\leq 2$. 
\end{itemize}
\end{theorem}

\begin{proof} ($\Rightarrow$) Since Cohen-Macaulay rings are unmixed, (i) is true. We shall prove (ii) by induction on $d$. 

We claim that $\C$ has a free vertex. As $I(\C)$ is 
normally torsion free, by \cite[Theorem~5.8]{reesclu}, there are minimal vertex covers 
$Z_1,\ldots,Z_d$ of $V(\C)$ such that $Z_1,\ldots,Z_d$ partition $V(\C)$ and $|Z_i\cap e|=1$ for every
$e\in E(\C)$ and $i=1,\dots,d$. Since $\C$ is
unmixed and since $Z_i$ is a minimal vertex cover of $\C$ for
all $i$, one has $|Z_i|=2$ for all $i$. It follows from
\cite[Corollary~4.14]{reesclu} that
$\C$ has a perfect matching, i.e., there are edges $e_1, e_2$
of $\C$ such that $e_1 \cap e_2=\emptyset$ and $e_1 \cup e_2 = V(\C)$. We may assume that $e_1=\{x_1,\ldots,x_d\}$, $e_2=\{y_1,\ldots,y_d\}$, and $Z_i=\{x_i,y_i\}$ for all $i$. Notice
that any minimal vertex cover $C$ of $\C$ has the form
$C=\{x,y\}$ for some $x\in e_1$ and $y\in e_2$. Let $G = \C^\vee$ be the Alexander dual of $\C$, which, in this case, is a graph.
The graph $G$ is bipartite with bipartition $e_1,e_2$. Since $R/I(\C)$ is Cohen-Macaulay, $I(G) = I(\C^\vee)$ has a linear resolution. It then follows from a result of Fr\"oberg \cite{Fro4} (see also \cite{Lyu1}) that the complement graph $G'$ of $G$ is chordal. By Remark \ref{dirac}, $G'$ has a simplicial vertex $z$. We may assume that $z=x_k$ for some $k$;
the case $z=y_k$ is symmetric. Observe that the induced subgraphs 
$G'_{e_1}$ and $G'_{e_2}$ of $G'$ on $e_1$ and $e_2$ are complete
graphs of size $d$. Next we prove that $x_k$ is not in 
$N_{G'}(e_2)$ for any $k$. If $x_k$ is in $N_{G'}(e_2)$ for some $k$, then
$\{x_k,y_\ell\}$ is an edge of $G'$ for some $\ell$. Consequently 
$y_\ell$ would have to be adjacent in $G'$ to any $x_i$ in $e_1$, in
particular $\{x_\ell,y_\ell\}\in E(G')$, a contradiction. Thus
$\{x_k,y_i\}\in E(G)$ for all $i$. Note that
$y_k$ is a free vertex of $\C$. Indeed let $e$ be any edge 
of $\C$ containing $y_k$, then $x_k$ is not in $e$ because
$|e\cap Z_k|=1$. Hence since $\{x_k,y_i\}$ is a vertex cover of $\C$ for any $i$ we
get that $y_i\in e$ for any $i$, i.e., $e=e_2$. 

Consider the edge
ideal $I'$ which is obtained from $I(\C)$ by making $x_k=1$
and $y_k=1$. Let $\C'$ be the clutter on $V' = V(\C) \setminus\{x_k,y_k\}$ that corresponds to $I'$, i.e., $I'=I(\C')$. The ideal $I'$ is Cohen-Macaulay of height two, normally torsion 
free, and is generated by monomials of degree $d-1$. 
Therefore, by the induction hypothesis, there is a partition
$X^2=\{x_1^2,x_2^2\}, \dots, X^d=\{x_1^d,x_2^d\}$ of $V'$ such that
all edges of $\C'$ have the form $\{x_{i_2}^2, \dots, x_{i_d}^d\}$ for some
$1\leq i_1\leq\dots\leq i_d\leq 2$. To complete the proof 
we set $x_1^1=x_k$, $x_2^1=y_k$, and $X^1=\{x_1^1,x_2^1\}$.

($\Leftarrow$) It follows from Theorem~\ref{g=2quotients}. Here, the
assumption that $I(\C)$ is normally torsion free is not
needed.

\end{proof}


\section{Cohen-Macaulay Clutters of Height Three} \label{g=3}

In this section, we prove Conjecture \ref{conjecture} in the case of uniform clutters with $g = \height I(\C) = 3$. Our method in this case, similar to the case of height 2, is to give an ordering for the set of all tuples $T = \{ (a,b,c) ~|~ 1 \le a,b,c \le d\}$, and then to show that $I(\C)^\vee$, whose minimal generators are labeled by the increasing order of their exponent vectors (induced by the ordering on $T$), admits linear quotients.

\begin{theorem} \label{g=3quotients}
Let $\C$ be a uniform admissible clutter of height 3. Then the Alexander dual $I(\C)^\vee$ of $I(\C)$ has linear quotients, and so $\C$ is a Cohen-Macaulay clutter.
\end{theorem}

\begin{proof} Since $\C$ is unmixed, it will be Cohen-Macaulay if $I(\C)^\vee$ has linear quotients, as noted in the Introduction. To prove $I(\C)^\vee$ has linear quotients, we extend the ordering of $S$ given in the proof of Theorem \ref{g=2quotients} to an ordering of the elements in $T$. We order $(a,b,c) < (d,f,h)$ if $(a,c) < (d,h)$ in $S$, and order $(a,b,c) < (a,b',c)$ if $b > b'$. Thus the elements of $T$ are ordered $(1,d,d) < (1,d-1,d) < (1,d-2,d) < \dots < (1,1,d) < (1,d,d-1) < (1,d-1,d-1) < (1,d-2,d-1) < \dots < (1,1,d-1) < (1,d,d-2) < \dots < (d,1,1)$. This induces an ordering on the exponent vectors of the minimal generators of $I(\C)^\vee$. We shall label the minimal generators of $I(\C)^\vee$ in the increasing order of their exponent vectors. Assume that the minimal generators of $I(\C)^\vee$ are labeled as $u_1, \dots, u_s$.

Suppose that for some $j < i$, we have $\deg u_j/\gcd(u_j,u_i) \ge 2$. Let $u_j = x^{j_1}_1x^{j_2}_2x^{j_3}_3$ and $u_i = x^{i_1}_1x^{i_2}_2x^{i_3}_3$. Then the exponent vectors $(j_1,j_2,j_3)$ and $(i_1,i_2,i_3)$ differ in at least two positions. Since $g=3$, this means that these two vectors differ in at least one of the two ends. If $j_1 \not= i_1$, then by the chosen ordering, $j_1 < i_1$. It follows from Lemma \ref{LowerTheFront} that $u_k = x_1^{j_1}x_2^{i_2}x_3^{i_3}$ is a minimal generator of $I(\C)^\vee$. By the chosen ordering, we also have $k < i$. Moreover, $x^{j_1}_1 = u_k/\gcd(u_k,u_i)$ is a linear factor of $u_j/\gcd(u_j,u_i)$.

Consider the case when $j_1 = i_1$. In this case, $j_l \not= i_l$ for $l=2,3$. By the chosen ordering, $j_3 > i_3$. Thus, it follows from Lemma \ref{RaiseTheEnd} that $u_k = x_1^{i_1}x_2^{i_2}x_3^{j_3}$ is a minimal generator of $I(\C)^\vee$. By the chosen ordering, we also have $k < i$. The conclusion follows from the fact that $x^{j_3}_3 = u_k/\gcd(u_k,u_i)$ is a linear factor of $u_j/\gcd(u_j,u_i)$. 
\end{proof}

Note that there are other orderings that work in the proof of Theorem \ref{g=3quotients} as well. For example, for $d=3$, a similar ordering given by $(a,b,c) < (d,f,h)$ if $(a,c) < (d,h)$ in the sequence $(1,3) < (1,2) < (2,3) < (2,2) < (1,1) < (2,1) < (3,3) < (3,2) < (3,1)$, and $(a,b,c) < (a,b',c)$ if $b > b'$, satisfies the two necessary requirements: raising the third entry and lowering the first entry of a generator of $I(\C)^\vee$ both result in a generator which occurs earlier in the list. This ordering will also give linear quotients in $I(\C)^\vee$. 

Both of the requirements above are necessary. We shall give an example of an unmixed uniform admissible clutter $\C$ satisfying $d=g=3$ for which the minimal vertex covers (which all have size $3$) do not have linear quotients under the reverse lexicographic ordering. Notice that you can raise the end in reverse lex, but if you lower the front, you might get an element that is higher in the overall ordering.

\begin{example}
Let $\C=\{ x_1y_1z_1, x_2y_2z_2, x_3y_3z_3, x_1y_2z_3 \} $, where to simplify notation, $x,y,z$ represent elements of $X^1,X^2,X^3$ respectively. There are 19 minimal vertex covers of order 3. Under the reverse lexicographic ordering, linear quotients fails. Indeed, let $C_i=\{ z_1, y_2, y_3\} $. It can be checked that $C_i$ is a cover. Also, $C_j= \{ x_1,z_2,y_3\} $ is also a cover, and $x_1z_2y_3 < z_1y_2y_3$, while the colon is $(x_1z_2y_3 : z_1y_2y_3) = (x_1z_2)$. Now $\{ z_1, z_2, y_3\} $ misses the edge $x_1y_2z_3$ of $\C$ and so is not a cover, so $((u_1, \ldots , u_{i-1}):u_i)$ will not contain $z_2$. Here, $u_l$'s are monomials corresponding to vertex covers $C_l$'s. Now $\{ x_1, y_2, y_3\} $ is a cover, but under the reverse lexicographic order, $u_k=x_1y_2y_3 > u_i=z_1y_2y_3$. Thus $x_1$ is also not in $(u_1, \ldots , u_{i-1}:u_i)$. Hence reverse lexicographic order will not suffice for $g=3$.
\end{example}


\section{Linear Quotients and Clutters of Higher Heights} \label{higherg}

In this section, we give a criterion in the case of uniform clutters of $\height I(\C) = 4$ under which an admissible unmixed clutter is Cohen-Macaulay, and present a family of examples to show that Conjecture \ref{conjecture} may fail when $\height I(\C) \ge 4$, even in the uniform case. 

We start by considering the case when $g = \height I(\C) = 4$. For convenience, we shall identify the vertex cover $C=(x_1^a,x_2^b,x_3^c,x_4^d)$ of $\C$ with the generator $u=x_1^ax_2^bx_3^cx_4^d$ of the Alexander dual $I(\C)^\vee$. That is, when we use the notation $(x_1^a,x_2^b,x_3^c,x_4^d)$ it is understood that we are talking about a vertex cover (as a set) of $\C$, and when we use the monomial notation $x_1^ax_2^bx_3^cx_4^d$ it is understood that we are talking about the same vertex cover but as a generator of the Alexander dual.

\begin{definition} \label{OrderCondition}
Let $\C$ be an unmixed uniform admissible clutter with $g=4$. We say that $\C$ {\bf satisfies condition (*)} if there exist two vertex covers $C_1 = (x_1^a, x_2^b, x_3^c, x_4^d)$ and $C_2 = (x_1^a, x_2^s, x_3^t, x_4^d)$ sharing the first and the last vertices so that neither $(x_1^a,x_2^b,x_3^t,$ $x_4^d)$ nor $(x_1^a,x_2^s,x_3^c,x_4^d)$ is a vertex cover of $\C$. In this case, we call the pair $(C_1,C_2)$ a {\bf bad vertex cover pair} of $\C$.
\end{definition}

\begin{lemma} \label{NotSymmetry}
Let $\C$ be an unmixed uniform admissible clutter with $g = 4$, and assume that $\C$ does not satisfy condition (*). Suppose $(x_1^a, x_2^b, x_3^c, x_4^d)$ and $(x_1^a, x_2^s, x_3^t, x_4^d)$ are vertex covers of $\C$ with $c > t$ such that $(x_1^a, x_2^s, x_3^c, x_4^d)$ is not a vertex cover of $\C$. Then, there does not exist a vertex cover $(x_1^a, x_2^s, x_3^q, x_4^d)$ of $\C$ with $q > t$ such that $(x_1^a, x_2^b, x_3^q, x_4^d)$ is not a vertex cover of $\C$.
\end{lemma}

\begin{proof} Suppose, by contradiction, that such a vertex cover $(x_1^a, x_2^s, x_3^q, x_4^d)$ of $\C$ exists. Then $(x_1^a,x_2^s,x_3^q,x_4^d)$ and $(x_1^a,x_2^b,x_3^c,x_4^d)$ form a bad vertex cover pair of $\C$, and so $\C$ satisfies condition (*), a contradiction.
\end{proof}

\begin{theorem} \label{thm.g=4}
Let $\C$ be an unmixed uniform admissible clutter with $g=4$. Then the Alexander dual $I(\C)^\vee$ of $I(\C)$ has linear quotients if and only if $\C$ does not satisfy condition (*). In particular, if $\C$ does not satisfy condition (*) then $\C$ is a Cohen-Macaulay clutter.
\end{theorem}

\begin{proof} The last statement of the theorem follows from our observations in the Introduction. We shall prove the first statement of the theorem. Suppose first that $\C$ satisfies property (*). We shall show that $I(\C)^\vee$ does not have linear quotients. Let $(C_1,C_2)$ be a bad vertex cover pair of $\C$, where $C_1 = (x_1^a,x_2^b,x_3^c,x_4^d)$ and $C_2 = (x_1^a,x_2^s,x_3^t,x_4^d)$. We shall use $u_1$ and $u_2$ to denote the corresponding monomials in $I(\C)^\vee$. Suppose there is an order of the minimal generators of $I(\C)^\vee$ that admits linear quotients. Without loss of generality, assume that $u_1 < u_2$ in this order. Since $u_1/\gcd(u_1,u_2) = x_2^bx_3^c$, in order to have linear quotients, there must exists a generator $u < u_2$ of $I(\C)^\vee$ such that $u/\gcd(u,u_2)$ is a linear factor of $x_2^bx_3^c$. Since $\C$ is unmixed, this implies that $u$ has to be either $x_1^ax_2^bx_3^tx_4^d$ or $x_1^ax_2^sx_3^cx_4^d$. This is a contradiction to the fact that $\C$ satisfies property (*) and $(C_1,C_2)$ is a bad vertex cover pair of $\C$.

Conversely, assume that $\C$ does not satisfy property (*). We shall construct an order of the generators of $I(\C)^\vee$ that admits linear quotients. We order these generators by the following rules:
\begin{enumerate}
\item $x_1^ax_2^bx_3^cx_4^d \prec x_1^mx_2^nx_3^px_4^q$ if $(a,d) < (m,q)$ in the order given in Theorem \ref{g=2quotients}, i.e., if $a < m$ or if $a=m$ and $d > q$,
\item $x_1^ax_2^bx_3^cx_4^d \prec x_1^ax_2^sx_3^tx_4^d$ if $c > t$,
\item $x_1^ax_2^bx_3^tx_4^d \prec x_1^ax_2^sx_3^tx_4^d$ if there exists $c > t$ such that $(x_1^a,x_2^b,x_3^c,x_4^d)$ is a vertex cover of $\C$, but $(x_1^a,x_2^s,x_3^c,x_4^d)$ is not, and
\item $x_1^ax_2^bx_3^tx_4^d \succ x_1^ax_2^sx_3^tx_4^d$ if such a $c$ as in (3) does not exist and $b < s$.
\end{enumerate}

We claim that this is a well-defined total order of the generators of $I(\C)^\vee$. Clearly, it suffices to prove that this partial order is well-defined (then it follows that the partial order is a total order). Indeed, rules (1) and (2) are well-defined. We only need to show that rules (3) and (4) are also well-defined. To this end, we show that the partial order given by rules (3) and (4) is anti-symmetric and transitive.

Suppose $u_1 = x_1^ax_2^bx_3^tx_4^d$ and $u_2 = x_1^ax_2^sx_3^tx_4^d$ are distinct generators of $I(\C)^{\vee}$ such that $u_1 \preceq u_2$ and $u_2 \preceq u_1$ by rules (3) and (4). We need to show that $u_1 = u_2$. Clearly, if both of the inequalities $u_1 \preceq u_2$ and $u_2 \preceq u_1$ are given by rule (4) then $b = s$, and hence, $u_1 = u_2$. Assume now that $u_1 \preceq u_2$ is given by the existence of $c > t$ such that $(x_1^a,x_2^b,x_3^c,x_4^d)$ is a vertex cover of $\C$, but $(x_1^a,x_2^s,x_3^c,x_4^d)$ is not. By the existence of $c$, $u_2 \preceq u_1$ cannot be given by rule (4). This implies that if $u_1 \not= u_2$ then $u_2 \preceq u_1$ is given by the existence of $q > t$ such that $(x_1^a, x_2^s, x_3^q, x_4^d)$ is a vertex cover of $\C$ but $(x_1^a, x_2^b, x_3^q, x_4^d)$ is not. However, the existence of such a $q$ contradicts the assertion of Lemma \ref{NotSymmetry}. Hence, $u_1 = u_2$. That is, the order $\prec$ is anti-symmetric.

Suppose $C_1 = (x_1^a,x_2^b,x_3^t,x_4^d), C_2 = (x_1^a,x_2^s,x_3^t,x_4^d)$ and $C_3 = (x_1^a,x_2^r,x_3^t,x_4^d)$ are vertex covers of $\C$ such that for the corresponding monomials we have $u_1 \prec u_2$ and $u_2 \prec u_3$. To get transitivity, we need to show that $u_1 \prec u_3$. Clearly, the partial order given by rule (4) is transitive. 

{\bf Case 1:} $u_1 \prec u_2$ and $u_2 \prec u_3$ are given by rule (3). That is, there exist $c, p > t$ such that
\begin{itemize}
\item $A = (x_1^a,x_2^b,x_3^c,x_4^d)$ and $B = (x_1^a,x_2^s,x_3^p,x_4^d)$ are vertex covers of $\C$, and
\item $C = (x_1^a,x_2^s,x_3^c,x_4^d)$ and $D = (x_1^a,x_2^r,x_3^p,x_4^d)$ are not vertex covers of $\C$.
\end{itemize}

Observe that since $C_2$ and $B$ are vertex covers of $\C$, it follows from Lemma \ref{Consecutive} that $(x_1^a,x_2^s,x_3^l,x_4^d)$ is a vertex cover of $\C$ for any $t \le l \le p$. Since $C$ is not a vertex cover of $\C$, we must have $c > p$. 

If $(x_1^a,x_2^r,x_3^c,x_4^d)$ is a vertex cover of $\C$ then by considering it together with the vertex cover $C_3$, it also follows from Lemma \ref{Consecutive} that $(x_1^a,x_2^r,x_3^l,x_4^d)$ is a vertex cover of $\C$ for any $t \leq l \leq c$. In particular, this implies that $D = (x_1^a,x_2^r,x_3^p,x_4^d)$ is a vertex cover of $\C$, a contradiction. Therefore, $(x_1^a,x_2^r,x_3^c,x_4^d)$ is not a vertex cover of $\C$. By rule (3), this implies that $u_1 \prec u_3$.

{\bf Case 2:} $u_1 \prec u_2$ is given by rule (3) and $u_2 \prec u_3$ is given by rule (4). If $u_3 \prec u_1$ by rule (4) then, since rule (4) is transitive, we have $u_2 \prec u_1$. This implies that $u_1 = u_2$, a contradiction. If $u_3 \prec u_1$ by rule (3) then by the same argument as in Case 1 above, we have $u_3 \prec u_2$. Again, this implies $u_2 = u_3$, a contradiction. Hence, we must have $u_1 \prec u_3$.

{\bf Case 3:} $u_1 \prec u_2$ by rule (4) and $u_2 \prec u_3$ by rule (3). We can use the same line of arguments as in Case 2 to conclude that $u_1 \prec u_3$.

We have shown that the order $\prec$ is transitive. Hence, $\prec$ gives a total order on the generators of $I(\C)^\vee$. It remains to show that under $\prec$, $I(\C)^\vee$ admits linear quotients.

Consider any two generators $U = x_1^ax_2^bx_3^cx_4^d$ and $V = x_1^mx_2^nx_3^px_4^q$ such that $U \prec V$. If $a \not= m$ then by our order (as given in the case $g = 2$), we have $a < m$. By Lemma \ref{LowerTheFront}, we have $W = x_1^ax_2^nx_3^px_4^q$ is a generator of $I(\C)^{\vee}$. Moreover, $W \prec V$ and $W/\gcd(W,V)$ is a linear factor of $U/\gcd(U,V)$. Assume that $a = m$. 

By a similar argument using Lemma \ref{RaiseTheEnd} in place of Lemma \ref{LowerTheFront}, if $d \not= q$, we can find a generator $W = x_1^mx_2^nx_3^px_4^d \prec V$ such that $W/\gcd(W,V)$ is a linear factor of $U/\gcd(U,V)$. Assume now that $a = m$ and $d = q$. 

If $c = p$ then we can take $W = U$ to have $W/\gcd(W,V)$ being linear, so we assume that $c > p$. If $c \le d$, then by using Lemma \ref{RaiseTheEnd}, we have a generator $W = x_1^ax_2^nx_3^cx_4^d \prec V$ such that $W/\gcd(W,V)$ is a linear factor of $U/\gcd(U,V)$. Thus, we may also assume that $c > d$. We now have $c > \max\{p,d\}$. 

If $x_1^ax_2^nx_3^cx_4^d$ is a generator of $I(\C)^{\vee}$ then by taking $W = x_1^ax_2^nx_3^cx_4^d$, we also have $W \prec V$ and $W/\gcd(W,V)$ is a linear factor of $U/\gcd(U,V)$. If $(x_1^a,x_2^n,x_3^c,x_4^d)$ is not a vertex cover of $\C$ then by rule (3), we have $W = x_1^ax_2^bx_3^px_4^d \prec V$ and $W/\gcd(W,V)$ is a linear factor of $U/\gcd(U,V)$. Moreover, in this case, since $\C$ does not satisfy condition (*), $W$ is a vertex cover of $\C$. 

Hence, we have shown that in any case, we can always find a generator $W$ of $I(\C)^{\vee}$ such that $W \prec V$ and $W/\gcd(W,V)$ is a linear factor of $U/\gcd(U,V)$. This shows that $I(\C)^\vee$, under the order $\prec$, admits linear quotients. The theorem is proved.
\end{proof}

We shall now give a family of counterexamples to Conjecture \ref{conjecture} when $g \ge 4$. 

\begin{lemma}[see {\cite[Proposition 6.2.7]{V}}] \label{reduce-g}
Let $S = K[x_1, \dots, x_n]$, $T = K[y_1, \dots, y_m]$, and $R = K[x_1, \dots, x_n, y_1, \dots, y_m]$. Suppose $I \subset S$ and $J \subset T$ are homogeneous ideals, and $IR$ and $JR$ are their extensions in $R$, respectively. Then $R/(IR, JR)$ is a Cohen-Macaulay ring if and only if both $S/I$ and $T/J$ are Cohen-Macaulay rings.
\end{lemma}

\begin{theorem} \label{counterexamples}
Let $g \ge 4$ be an integer. Then there always exists a uniform, admissible and unmixed clutter $\C$ of height $g$ that is not Cohen-Macaulay.
\end{theorem}

\begin{proof} We shall construct such a clutter $\C$ explicitly. Let $R = k[\x,\y,\z,\w,\u]$, where $\x = (x_1, \dots, x_g), \y = (y_1, \dots, y_g), \z = (z_1, \dots, z_g), \w = (w_1, \dots, w_g), \u = (u_1, \dots, u_g)$. Take $\C$ to be the clutter over the vertices $\x \cup \y \cup \z \cup \w \cup \u$ with edge set
$$E(\C) = \langle e_1, \dots, e_g, x_1y_2z_3w_3u_4, x_1y_1z_2w_2u_3, x_1y_1z_3w_3u_3, 
x_1y_2z_2w_2u_4 \rangle$$
where $e_i = x_iy_iz_iw_iu_i$ for all $i = 1, \dots, g$ (here, by abusing notation, we identify an edge with the corresponding monomial). 

By construction, $\C$ is uniform. By verifying with conditions in Definition \ref{def.admissible}, it can be seen that $\C$ is admissible. To prove the unmixedness of $\C$ we need to show that if $C$ is a minimal vertex cover of $\C$ then $|C| = g$. We first have $|C \cap e_i| \ge 1$ for all $i = 1, \dots, g$. It also follows from the minimality of $C$ that $|C \cap e_i| = 1$ for $i \ge 5$ (since $e_i$ is the only edge of $\C$ involving the vertices $\{x_i, y_i, z_i, w_i, u_i\}$). Observe that if $C$ is a minimal vertex cover of $\C$, then $C \cap \{x_1, \dots, x_4, y_1, \dots, y_4, \dots, u_1, \dots, u_4 \}$ is a minimal vertex cover of the clutter consisting of edges $$\{e_1, e_2, e_3, e_4, x_1y_2z_3w_3u_4, x_1y_1z_2w_2u_3, x_1y_1z_3w_3u_3, 
x_1y_2z_2w_2u_4\}.$$
Thus, by a direct computation (either with CoCoA \cite{cocoa} or Macaulay 2 \cite{macaulay2}), we further have $|C \cap e_i| = 1$ for $i=1,2,3,4$. Therefore, $|C| = g$.

Finally, we show that $\C$ is not a Cohen-Macaulay clutter. Let 
$S = K[x_1, \dots, x_4,$ $y_1, \dots, y_4, \dots, u_1, \dots, u_4], T = K[x_5, \dots, x_g, y_5, \dots, y_g, \dots, u_5, \dots, u_g],$
and let 
$$I = \langle e_1, e_2, e_3, e_4, x_1y_2z_3w_3u_4, x_1y_1z_2w_2u_3, x_1y_1z_3w_3u_3, 
x_1y_2z_2w_2u_4 \rangle \subset S$$ 
and $J = \langle e_5, \dots, e_g \rangle \subset T.$
Observe that $I(\C) = IR + JR$. By a direct computation (either with CoCoA \cite{cocoa} or Macaulay 2 \cite{macaulay2}), we have $S/I$ is not a Cohen-Macaulay ring. It now follows from Lemma \ref{reduce-g} that $R/I(\C)$ is not a Cohen-Macaulay ring. That is, $\C$ is not a Cohen-Macaulay clutter.
\end{proof}

\begin{remark} When $g=4$, the clutter $\C$ constructed in Theorem \ref{counterexamples} satisfies condition (*).
\end{remark}


\end{document}